\documentclass[12pt]{article}
\usepackage[margin=1in]{geometry}
\DeclareMathAlphabet\mathbfcal{OMS}{cmsy}{b}{n}
\usepackage[utf8]{inputenc}
\usepackage[dvipsnames]{xcolor}
\usepackage{enumerate}
\usepackage[utf8]{inputenc}
\usepackage{amsfonts}
\usepackage{amsthm}
\usepackage{amsmath}
\usepackage{lscape}
\usepackage{mathrsfs}
\usepackage{graphicx}
\usepackage{tikz-cd}
\usepackage{calc}
\usepackage{float}
\usepackage{stmaryrd}
\usepackage{multirow}
\usepackage{lipsum}
\usepackage{amssymb}
\usepackage{mathrsfs}
\usepackage{mathtools}
\usepackage{mathdots}
\usepackage{fancyhdr}
\usepackage{tikz}
\usepackage{faktor}
\usepackage{setspace}
\usetikzlibrary{decorations.markings}
\newtheorem{theorem}{Theorem}
\theoremstyle{plain}
\usepackage{sectsty}
\usepackage{kpfonts}
\newtheorem{corollary}{Corollary}
\newtheorem{definition}{Definition}
\newtheorem{example}{Example}

\newtheorem*{proposition*}{Proposition}

\newcommand{\Hom}{{\mbox{Hom}}}

\newcommand{\vol}{{\mbox{vol}}}

\newtheorem*{observation*}{Observation}
\newtheorem*{theorem*}{Theorem}
\newtheorem*{claim*}{Claim}  
\usepackage{blindtext}
\usepackage{fancyhdr}
\title{Geometric structures as variational objects, I}
\author{Gabriella Clemente}
\date{}
\allowdisplaybreaks
\makeatletter
\renewcommand\tableofcontents{%
    \@starttoc{toc}%
}
\begin{document}
\maketitle

\begin{abstract}
This note provides a variational description of the most basic differential geometric structures on a smooth manifold.
\end{abstract}

\section*{Introduction}
Symplectic forms, complex structures, and K{\"a}hler complex structures on a manifold $M$ share the common trait of being smooth sections of a vector bundle on $M,$ satisfying a homogeneous PDE system in their coefficients. They are colloquially referred to as differential geometric structures on $M.$ This note formulates in a precise way the concept of what might be the simplest geometric structures, and shows that they can always be concretely realized via functionals. The main observation is that these simple geometric structures are variational objects (see Theorem \ref{T1}), a statement that can be understood intuitively at this point, but which will acquire a crystalized meaning in the next pages. Symplectic manifolds, Riemannian manifolds with exceptional holonomy, and Calabi-Yau $3$-folds have been defined via functionals in \cite{Hitchin}. 

The intention is to continue the study of these geometric structures in forthcoming work that will cover more intricate examples, such as the $\alpha$-integrable almost-complex structures from \cite{CSACPS}. The long-term objective is to reverse engineer Theorem \ref{T1}, and its future generalizations, to arrive at a ``motivic architecture" of sorts that encapsulates those aspects of existence problems about geometric structures that are systematic. In other words, this architecture should be a set of principles that make up the core of variational phenomena in the realm of geometric structures. The notions of geometric flows and stability (in the sense of K-stability, and see \cite{CSACPS} for a formal, almost-complex analog) are candidate idea-pillars of this mindset. 

The first section of this note sets up the machinery for the calculus of variations of vector bundle valued differential forms that is used to write down the functional that characterizes geometric structures, and which is the theme of section 2. The conclusion attempts to discuss geometric structures from a more abstract point of view.

\section{Bundle forms and geometric structures}
Let $M$ be a smooth manifold of real dimension $n,$ and $E \to M$ be a real vector bundle of rank $m.$ Consider the space of $E$-valued differential forms on $M,$ \[\Omega^{\bullet} (M, E)=\bigoplus_{k=0}^n \Omega^k (M, E).\] For $E=M \times \mathbb{R}$ the trivial line bundle, the $E$-valued forms coincide with the usual differential forms $\Omega^{\bullet} (M)=\bigoplus_{k=0}^n \Omega^k (M).$ An element $\gamma \in \Omega^{\bullet} (M, E)$ decomposes as a sum $\gamma=\sum_{k=0}^n \gamma_k.$ 

\subsection{Bundle form calculus}
Let $h_E$ be a metric on $E.$ Then, $\Omega^{\bullet} (M, E)$ can be endowed with an $L^2$-inner product as explained in the sequel. For any $k,l$ with $0 \leq k+l \leq n,$ consider the bilinear map \[\wedge_{h_E}:\Omega^k (M, E) \otimes \Omega^l (M, E) \to  \Omega^{k+l} (M),\] defined on pure elements as $(a \otimes v) \wedge_{h_E} (b \otimes w)=a \wedge b h_E (v,w).$ In a local frame $(\epsilon_i)_{i=1}^m$ of $E,$ where $\alpha=\sum_{i=1}^m a_i \otimes \epsilon_i \in \Omega^k (M, E),$ $\beta=\sum_{j=1}^m b_j \otimes \epsilon_j \in \Omega^l (M, E),$ \[\alpha \wedge_{h_E} \beta=\sum_{i,j=1}^m a_i \wedge b_j h_E (\epsilon_i,\epsilon_j).\] The Riemannian metric induces an inner product on exterior powers of $T^*_M,$ which when combined with the bundle metric on $E,$ gives an inner product \[\langle \cdot, \cdot \rangle_{h_E}:  \Omega^k (M, E) \otimes \Omega^k (M, E) \to \mathbb{R}, \quad \langle a \otimes v, a' \otimes v' \rangle_{h_E}=\langle a,a'\rangle h_E(v,v').\] Locally, if $\alpha=\sum_{i=1}^m a_i \otimes \epsilon_i, \alpha'=\sum_{j=1}^m a'_j \otimes \epsilon_j \in \Omega^k (M, E),$ then \[\langle \alpha, \alpha' \rangle_{h_E}=\sum_{i,j=1}^m \langle a_i,a'_j\rangle h_E(\epsilon_i,\epsilon_j).\] Thanks to the metric on $E,$ the usual Hodge star operator on $\Omega^{\bullet} (M)$ extends to an operator $\star_{h_E}:\Omega^{\bullet} (M, E) \to \Omega^{\bullet} (M, E)$ with the defining property that for any $\alpha, \alpha' \in \Omega^k (M, E),$ \[\alpha \wedge_{h_E} \star_{h_E} \alpha'=\langle \alpha, \alpha'\rangle_{h_E} \vol_g,\] where $\vol_g$ is the Riemannian volume form of $(M,g).$ The notation $\langle \cdot, \cdot \rangle_{h_E}$ is being reserved for the inner product on $\Omega^k (M, E)$ only, and it should be distinguished from $h_E (\cdot, \cdot),$ which is the metric inner product on $E.$ 

For each $k,$ there is a pairing \[\langle \cdot, \cdot\rangle^{h_E}_k:\Omega^k (M, E) \otimes \Omega^k (M, E) \to \mathbb{R}, \quad \langle \alpha, \beta\rangle^{h_E}_k=\int_M \alpha \wedge_{h_E} \star_{h_E} \beta,\] and this gives the $L^2$-inner product \[\langle \langle \cdot, \cdot \rangle \rangle_{h_E}:\Omega^{\bullet} (M, E) \otimes \Omega^{\bullet} (M, E) \to \mathbb{R}, \quad \langle \langle A, B \rangle\rangle_{h_E}:=\sum_{k=0}^n \langle A_k,B_k\rangle^{h_E}_k,\] where $A=\sum_{k=0}^n A_k, B=\sum_{k=0}^n B_k,$ and $A_k, B_k \in \Omega^k (M, E).$ If $E=T_M$ with $h_E=g,$ then \[\langle \langle \cdot, \cdot \rangle \rangle_{g}:\Omega^{\bullet} (M, T_M) \otimes \Omega^{\bullet} (M, T_M) \to \mathbb{R}, \quad \langle \langle A, B \rangle\rangle_{g}=\sum_{k=0}^n \langle A_k,B_k\rangle^{g}_k\] is the $L^2$-inner product from \cite{ACOTI}.

Observe that $g$ induces a metric on any vector bundle $\mathcal{E}(T_M)$ built from $T_M, T^*_M,$ and the resulting inner product on $\Omega^{\bullet} \big(M, \mathcal{E}(T_M)\big)$ will be denoted too by $\langle \langle \cdot, \cdot \rangle \rangle_{g}.$ For instance, consider the case $E=\bigwedge^{\bullet} T_M,$ and the inner product on $\Omega^{\bullet} \Big(M, \bigwedge^{\bullet} T_M \Big)$ that is discussed in \cite{CSACPS}. 

The notation for the usual Hodge pairing of $k$-forms used here is $\langle \langle \cdot, \cdot \rangle \rangle_k,$ where recall $\langle\langle a,b \rangle\rangle_k=\int_M a \wedge \star b=\int_M \langle a,b\rangle \vol_g.$ It gives an $L^2$-inner product \[\langle \langle \cdot, \cdot \rangle \rangle:\Omega^{\bullet} (M) \otimes \Omega^{\bullet} (M) \to \mathbb{R}, \quad \langle \langle A, B \rangle\rangle=\sum_{k=0}^n \langle\langle A_k,B_k\rangle\rangle_k,\] where $A_k, B_k \in \Omega^k (M),$ that is consistent with all of the above.

Now, let $\nabla^E$ be a linear connection on $E$; i.e.\ $\Omega^0(M,E) \to \Omega^1(M,E)$ is a linear map that satisfies the Leibniz rule: for any smooth function $\sigma \in C^{\infty}(M, \mathbb{R})$ on $M,$ and any smooth section $s \in \Omega^0(M,E)$ of $E,$ $\nabla^E (\sigma s)=d\sigma \otimes s+f \nabla^E s.$

The exterior covariant derivative $d^{\nabla^E}$ associatd with $\nabla^E$ is the degree $1$ -- $d^{\nabla^E}:\Omega^k(M,E) \to \Omega^{k+1}(M,E)$ -- operator on $E$-valued forms, which on $k$-forms $\alpha \in \Omega^k(M,E),$ acts as 

\begin{equation*}
\begin{split}
(d^{\nabla^E} \alpha)(\zeta_0,\dots,\zeta_k)&=\sum_{i=0}^k (-1)^i {\nabla^{E}}_{\zeta_i} \alpha(\zeta_0,\dots,\widehat{\zeta_i},\dots,\zeta_k)+\\
&\sum_{0 \leq i <j \leq k} (-1)^{i+j} \alpha ([\zeta_i,\zeta_j],\dots,\widehat{\zeta_i},\dots,\widehat{\zeta_j},\dots,\zeta_k).
\end{split}
\end{equation*}
The formal adjoint of $d^{\nabla^E}$ w.r.t.\ $\langle \langle \cdot , \cdot \rangle \rangle_{h_E}$ will be denoted here by $\delta^{\nabla^E}.$ For the trivial bundle $E=M \times \mathbb{R},$ $d^{\nabla^E}$ becomes the exterior derivative $d.$ Its adjoint w.r.t.\ to $\langle \langle \cdot, \cdot \rangle \rangle$ will be denoted by $\delta.$

\newpage 

\subsection{Degree $k$ geometric structures}
The simplest differential geometric structure on a manifold are solutions to homogeneous first order linear PDEs.

\begin{definition}\label{D1}
A geometric structure on $M$ of degree $k$ (or a $k$-geometric structure, for short) is a form $\Phi \in \Omega^k(M,V)$ such that $d^{\nabla^V} \Phi=0,$ where $V \to M$ is any real vector bundle, and $\nabla^V$ is any linear connection on $V.$ The $k$-geometric structures on $M$ of a specific kind a priori tend to belong to some subset $U \subseteq \Omega^k(M,E),$ hence are collectively describable as $U \cap \ker{d^{\nabla^V}}.$ 
\end{definition}
This is illustrated in the example below.

\begin{example}{(A non-exhaustive list of $k$-geometric structures)}\label{ex1}
\quad 

\begin{enumerate}
\item \underline{Symplectic forms}. Let \[U=\{\omega \in \Omega^2(M) \mid \omega \mbox{ is non-degenerate }\}.\] The symplectic structures on $M$ are $2$-geometric structures describable as \[U \cap \ker{d}.\] 

\item \underline{K{\"a}hler complex structures}. Let $(M,g)$ be a Riemannian manifold with Levi-Civita connection $\nabla^{LC},$ \[AC(M):=\{J \in \Omega^1 (M, T_M) \mid J \circ J=-Id\}\] be the space of almost-complex structures on $M,$ and \[AC(M)_g:=\{J \in AC(M) \mid g(J\zeta,J\eta)=g(\zeta,\eta), \forall \zeta, \eta \in \mathfrak{X}(M)\}\] be the subspace of $g$-orthogonal almost-complex structures. By viewing \[AC(M) \subset \Omega^0 (M,T^*_M \otimes T_M),\] $d^{\nabla^{LC}} J=\nabla^{LC} J,$ for any $J \in AC(M).$  An almost-complex structure $J \in AC(M)_g$ is K{\"a}hler iff $d^{\nabla^{LC}} J=0.$ So these are $0$-geometric structures. Let \[U=AC(M)_g \subset \Omega^0 (M,T^*_M \otimes T_M)\] to find that \[U \cap \ker{d^{\nabla^{LC}}}\] describes the K{\"a}hler complex structures on $(M,g).$

\item \underline{Special complex structures}. Let \[C(M) \subset AC(M) \subset \Omega^1(M,T_M)\] be the space of complex structures on $M.$ Consider the family of those $J \in C(M)$ such that $d^{\nabla} J=0$ for a fixed choice of torsion-free connection $\nabla$ on $T_M.$ Similar complex structures arise in the definition of \emph{special K{\"a}hler structures} (see \cite{Free}). But they do not always exist; e.g.\ when $\dim_{\mathbb{R}} M\geq 4,$ $g$ has non-zero constant sectional curvature, and the chosen connection is the Levi-Civita one (see Theorem 1 in \cite{ACOTI}). These families of $d^{\nabla}$-closed complex structures are $1$-geometric structures, and can be described as \[U \cap \ker{d^{\nabla}},\] where \[U=C(M) \subset \Omega^1(M,T_M).\]
\end{enumerate}
\end{example}

\section{Variational nature}
This section demonstrates how geometric structures correspond to functionals on spaces of vector bundle valued forms.

\begin{definition}\label{d2}
Let $p_k:\Omega^{\bullet} (M, E) \to \Omega^k (M, E)$ be the projection mapping $p_k (\gamma)=\gamma_k.$ A variational object of degree $k$ (or $k$-variational object) is the $kth$ projection $p_k(c)$ of a critical point $c$ of some functional $\mathcal{F}$ with $dom(\mathcal{F}) \subseteq \Omega^{\bullet} (M, V),$ where $V \to M$ is a real vector bundle. 
\end{definition}

Denote the restriction $d^{\nabla^E}\Big|_{\Omega^k (M, E)}$ by $d_k^{\nabla^E}:\Omega^k (M, E) \to \Omega^{k+1} (M, E),$ and likewise let $\delta^{\nabla^E}\Big|_{\Omega^k (M, E)}$ be denoted by $\delta_k^{\nabla^E}:\Omega^k (M, E) \to \Omega^{k-1} (M, E).$ For any $k,$ define an operator on $\Omega^{\bullet} (M, E),$ \[d^{\nabla^E}[k]:=\sum_{i=0}^{n-1}(1-\delta_{k-1}^i)d_i^{\nabla^E}\] so that 
\[d^{\nabla^E}[k]\Big|_{\Omega^i (M, E)}= \begin{cases} 
      d_i^{\nabla^E}, & i \neq k-1 \\
      0 , & i = k-1.
   \end{cases}
\]

This has a formal adjoint in the setup laid out in the previous section, $\delta^{\nabla^E}[k],$ where  
\[\delta^{\nabla^E}[k]\Big|_{\Omega^i (M, E)}= \begin{cases} 
      \delta_i^{\nabla^E}, & i \neq k \\
      0 , & i = k.
   \end{cases}
\]
If $E$ is the trivial line bundle on $M,$ then $d^{\nabla^E}[k]$ indeed specializes to \[d[k]=\sum_{i=0}^{n-1}(1-\delta_{k-1}^i)d_i:\Omega^{\bullet} (M)\ \to \Omega^{\bullet} (M)\] with adjoint $\delta[k].$  

\newpage

\begin{theorem}\label{T1}
Geometric structures of degree $k$ are $k$-variational objects.
\end{theorem}

\begin{proof}
Suppose that the geometric structures of interst are given as $U \cap \ker{d^{\nabla^E}}$ for $E \to M$ a vector bundle with connection $\nabla^E,$ and $U \subset \Omega^k(M,E).$ The aim is to show that there is a  functional $L$ with domain contained in $\Omega^{\bullet}(M,E)$ such that if $S$ is its set of critical points, then $U \cap \ker{d^{\nabla^E}}=p_k(S)$ (cf.\ Definitions \ref{D1} and \ref{d2}). 

Consider the extensions \[\widetilde{\Omega}^k (M,E):=\{\gamma \in \Omega^{\bullet}(M,E) \mid \delta^{\nabla^E} \gamma_{k+2}=0, \mbox{ and } \gamma_{k-2}=0\},\] and 
\[\widetilde{U}:=\{\gamma \in \widetilde{\Omega}^k (M,E) \mid \gamma_k \in U\}\] of $\Omega^k(M,E),$ and $U$ across $\Omega^{\bullet}(M,E).$ 

Let $h_E$ be a metic on $E,$ and define a functional \[\mathcal{F}^{\nabla^E}:\widetilde{\Omega}^k (M,E) \to \mathbb{R}, \quad \mathcal{F}^{\nabla^E}:=\langle\langle d^{\nabla^E}[k] \gamma,\gamma\rangle\rangle_{h_E}.\] The first variation is \[\frac{d}{dt}\Big|_{t=0} \mathcal{F}^{\nabla^E}(\gamma+t\beta)=\langle\langle (d^{\nabla^E}[k]+\delta^{\nabla^E}[k]) \gamma,\beta \rangle\rangle_{h_E},\] so the critical point set is 

\begin{equation*}
\begin{split}
\mathcal{S}&=\{\gamma \in \widetilde{\Omega}^k (M,E) \mid d^{\nabla^E}[k]\gamma_{l-1}+\delta^{\nabla^E}[k] \gamma_{l+1}=0, \forall 0 \leq l \leq n\}\\
&=\{\gamma \in \widetilde{\Omega}^k (M,E) \mid d^{\nabla^E}[k]\gamma_{l-1}+\delta^{\nabla^E}[k] \gamma_{l+1}=0, \forall 0 \leq l \leq n, l\neq k-1,k+1,\\
&d^{\nabla^E}[k]\gamma_{k-2}+\delta^{\nabla^E}[k] \gamma_{k}=0, \mbox{ and } d^{\nabla^E}[k]\gamma_k+\delta^{\nabla^E}[k] \gamma_{k+2}=0\}\\
&=\{\gamma \in \widetilde{\Omega}^k (M,E) \mid d^{\nabla^E}[k]\gamma_{l-1}+\delta^{\nabla^E}[k] \gamma_{l+1}=0, \forall 0 \leq l \leq n, l\neq k+1,\\
&d^{\nabla^E}[k]\gamma_k=0\}.
\end{split}
\end{equation*}
Note that by design, $d^{\nabla^E}[k],$ and $\delta^{\nabla^E}[k]$ eliminate the codifferential condition on $\gamma_k$ by making the equation $d^{\nabla^E}[k]\gamma_{k-2}+\delta^{\nabla^E}[k] \gamma_{k}=0$ redundant.

The sought after functional $L$ is $\mathcal{F}^{\nabla^E}\Big|_{\widetilde{U}},$ and it has critical point set $\mathcal{S} \cap \widetilde{U}.$ The proof is complete since \[p_k(\mathcal{S} \cap \widetilde{U})=\{\alpha \in U \mid d^{\nabla^E} \alpha=0\}=U \cap \ker{d^{\nabla^E}}.\]
\end{proof}

It follows at once that
\begin{corollary}\label{c1}
Symplectic forms, K{\"a}hler complex structures, and special complex structures (cf.\ Example \ref{ex1}), to name a few, are $k$-variational objects. The key information is summarized in Table \ref{tab1}.
\end{corollary}

\begin{landscape}
\begin{center}
\begin{table}
\centering
  \begin{tabular}{ l | c | c | c | c | r }
    \hline
    geometric structure & degree $k$ & vector bundle $E$ & connection $\nabla^E$ & $U \subseteq \Omega^k(M,E)$ & functional $\mathcal{F}^{\nabla^E}$ \\ \hline
    symplectic forms & 2 & $M \times \mathbb{R}$ & & $\omega$ non-degenerate & $\mathcal{S}(\gamma)=\langle \langle d[2] \gamma,\gamma\rangle \rangle$  \\ \hline
    K{\"a}her complex structures & 0 & $T^*_M \otimes T_M$ & $\nabla^{LC}$ & $AC(M)_g$ &  $\mathcal{K}(\gamma)=\langle \langle d^{\nabla^{LC}}[0] \gamma,\gamma\rangle \rangle_g$  \\ \hline
    special complex structures & 1 & $T_M$ & torsion-free $\nabla$ & $C(M)$ &  $\mathcal{C}^{\nabla}(\gamma)=\langle \langle d^{\nabla}[1] \gamma,\gamma\rangle \rangle_g$  \\ \hline
  \end{tabular}
  \caption{Functional counterpart of Example \ref{ex1}}\label{tab1}
  \end{table}
\end{center} 

\begin{center}
\begin{table}
\Large
\centering
  \begin{tabular}{ c | c }
    \hline
    geometric structure & flow \\ \hline
    symplectic forms & $\frac{\partial \gamma}{\partial t}=-(d[2]+\delta[2])\gamma$ \\ \hline
    K{\"a}her complex structures & $\frac{\partial \gamma}{\partial t}=-(d^{\nabla^{LC}}[0]+\delta^{\nabla^{LC}}[0])\gamma$  \\ \hline
    special complex structures & $\frac{\partial \gamma}{\partial t}=-(d^{\nabla}[1]+\delta^{\nabla}[1])\gamma$ \\ \hline
  \end{tabular}
  \caption{Formal flows}\label{table2}
  \end{table}
\end{center}
\end{landscape}

It is also immediate that

\begin{corollary}\label{c2}
The formal gradient flow of the functional $\mathcal{F}^{\nabla^E}=\langle \langle d^{\nabla^E}[k] \gamma,\gamma\rangle \rangle_{h_E}$ from the above proof is \[\frac{\partial \gamma}{\partial t}=-(d^{\nabla^E}[k]+\delta^{\nabla^E}[k])\gamma.\] 
\end{corollary}

Table \ref{table2} contains some flow examples. 

The combinatorial functional analysis that lurks beneath this bundle form calculus of variations is an interesting topic in its own right. A deeper understanding of it could lead to improvements of Theorem \ref{T1}. 

The following example is anomalous, but perhaps still worth recording. For a Riemannian metric $g$ on $M,$ and an almost-complex structure $J,$ put $\Omega_J(\cdot,\cdot)=g(J\cdot ,\cdot).$ Recall that the \emph{almost-K{\"a}hler structures on} $(M,g)$ can be implicitly defined as \[\{J \in AC(M)_g \mid d\Omega_J=0\},\] while the \emph{K{\"a}hler structures on} $(M,g)$ are \[\{J \in AC(M)_g \cap C(M) \mid d\Omega_J=0\}.\] See Example \ref{ex1} for the terminology in use. This is a misfit case since neither of these sets is expressible as $U \cap \ker{d}$ since $AC(M)_g \not\subset \Omega^{\bullet} (M).$ Nevertheless, one can describe these structures via functionals. To that end, for any $\gamma \in \widetilde{\Omega}^2 (M),$ and any $J \in AC(M)_g,$ set
\begin{equation*}
\begin{split}
\gamma_J&:=\gamma_1(J\cdot)+\gamma_2(J\cdot, \cdot)+\dots+\gamma_n (J\cdot,\dots,\cdot)\\
&=: (\gamma_J)_1+(\gamma_J)_2+\dots+(\gamma_J)_n.
\end{split}
\end{equation*}
Notice that $\gamma_{J+tV}=\gamma_J+t\gamma_V.$ Now, pick a $\gamma$ that produces non-tautological extensions $\gamma_J$ of $\Omega_J=(\gamma_J)_2,$ $J\in AC(M)_g.$ Let \[\mathcal{F}_g^{\gamma}:AC(M)_g \to \mathbb{R}, \quad \mathcal{F}_g^{\gamma}(J)=\langle \langle d[2] \gamma_J,\gamma_J\rangle \rangle.\] The set of critical points of this functional is \[S^{\gamma}=\{J\in AC(M)_g \mid d[2] (\gamma_J)_{k-1}+\delta[2](\gamma_J)_{k+1}=0, \forall k \neq 3, \mbox{and } d\Omega_J=0\}.\] and that of $\mathcal{F}_g^{\gamma}\big|_{AC(M)_g \cap C(M)}$ is $\mathcal{S}^{\gamma} \cap C(M).$ The (almost-)K{\"a}hler structures can now be seen to be represented by $S^{\gamma},$ respectively $\mathcal{S}^{\gamma} \cap C(M).$ The associated formal flow is 
\[\frac{\partial \gamma_J}{\partial t}=-(d[2]+\delta[2])\gamma_J.\]

\newpage

\section*{Conclusion}

Some natural questions to ask is are if $k$-geometric structures form a category, and if the assignment $(E,h_E,\nabla^E,U)\mapsto \mathcal{F}^{\nabla^E}\big|_{\widetilde{U}},$ or some other map encoding the association furnished by Theorem \ref{T1}, can be viewed as a functor. Below is a discussion pertaining only to the first question. Riemannian geometry seems to provide a sample pair of objects, and a morphism between them. Speculations on the category of $k$-geometric structures on $M,$ called here $GS_k(M),$ stem from this example. 

Let $g, g'$ be Riemannian metrics on $M$ with corresponding Levi-Civita connection $\nabla^{LC},$ $\nabla^{LC'}.$ Suppose that $U \cap \ker{d^{\nabla^{LC}}},$ $U' \cap \ker{d^{\nabla^{LC'}}}$ describe 2 different kinds of $k$-geometric structures on $M.$ Let $f:(M,g) \to (M,g')$ be an isometry such that $(Id_{\Omega^k(M)} \otimes f_*)(U)\subset U',$ where $f_*:T_M \to T_M$ is the pushforward of $f.$ Regarding $(T_M, g, \nabla^{LC},U),$ $(T_M, g', \nabla^{LC'},U')$ as potential objects of $GS_k(M),$ $f_* \in \Hom_{GS_k(M)}\big((T_M, g, \nabla^{LC},U),$ $(T_M, g', \nabla^{LC'},U')\big)$ seems to be a morphism that satisfies:

\begin{enumerate}
\item (naturality of the connection) $f_* (\nabla^{LC}_{\zeta} \eta)=\nabla^{LC'}_{f_{*} \zeta} (f_{*} \eta),$

\item (isometry condition) 
\begin{equation*}
\begin{split}
g(\zeta,\eta)&=(f^* g')(\zeta,\eta)\\
&=g'(f_{*} \zeta, f_{*} \eta)\\
&=\big((f^{\vee}_{*} \otimes f^{\vee}_{*})\circ g'\big)(\zeta, \eta),
\end{split}
\end{equation*} 
where $f^{\vee}_{*}:T^*_M \to T^*_M$ is the dual homomorphism, and
\item (assumption) $(Id_{\Omega^k(M)} \otimes f_*)(U)\subset U'.$
\end{enumerate} 

So could $k$-geometric structures on $M$ form a category with objects and morphisms akin to

\begin{equation*}
\begin{split}
ob GS_k(M)=&\mbox{ quadruples } (E,h_E,\nabla^E,U), \mbox{ where }E \to M \mbox{ is a real vector bundle with metric }h_E\\
& \mbox{ and connection }\nabla^E, \mbox{ and } U \subseteq \Omega^k(M,E) \mbox{ a subset}
\end{split}
\end{equation*}

\begin{equation*}
\begin{split}
mor GS_k(M)=&\Hom_{GS_k(M)}\big((E,h_E,\nabla^E,U), (F,h_F,\nabla^F,V)\big) \mbox{ over all pairs }\\
&(E,h_E,\nabla^E,U), (F,h_F,\nabla^F,V) \in ob GS_k(M),
\end{split}
\end{equation*}
where \[\Hom_{GS_k(M)}\big((E,h_E,\nabla^E,U), (F,h_F,\nabla^F,V)\big)\] consists of vector bundle homomorphisms $\sigma:E \to F$ such that
\begin{enumerate}
\item $\forall s \in \Omega^0(M,E),$ $\forall \zeta \in \mathfrak{X}(M),$ $\sigma(\nabla^E_{\zeta} s)=\nabla^F_{\zeta} (\sigma \circ s);$

\item $(\sigma^{\vee} \otimes \sigma^{\vee}) \circ h_F=h_E,$ where $\sigma^{\vee}$ is the dual of $\sigma;$ and 
\item $(Id_{\Omega^k(M)} \otimes \sigma)(U)\subset V$ \quad ?
\end{enumerate}

\noindent
Gabriella Clemente

\noindent
e-mail: clemente6171@gmail.com
\end{document}